\documentclass[11pt]{article}
\usepackage{latexsym,amsmath,stackrel,color,amsthm,amssymb,epsfig,graphicx,mathrsfs}
\usepackage{graphicx}
\usepackage{amssymb}
\usepackage[left=0.9in,top=0.9in,right=0.9in,bottom=0.9in]{geometry}
\usepackage[linktocpage=true]{hyperref}
\usepackage{setspace}
\usepackage{amssymb, amsmath, amsthm, graphicx,mathrsfs}
\usepackage{comment}
\usepackage{caption}

\usepackage{tikz}
\makeatletter
\def\thm@space@setup{%
  \thm@preskip=\parskip \thm@postskip=0pt
}
\makeatother
\def\qed{\hfill\ifhmode\unskip\nobreak\fi\quad\ifmmode\Box\else\hfill$\Box$\fi}
\def\ite#1{\hfill\break${}$\hbox to 50pt {\quad(#1)\hfill}}

\newtheorem{thm}{Theorem}[section]

\newtheorem{definition}{Definition}

\newtheorem{lem}[thm]{Lemma}
\newtheorem{conjecture}{Conjecture}

\def\ex{{\rm{ex}}}
\def\ord{{\rm{ord}}}

\begin{document}

\pagestyle{myheadings}
\markright{{\small{\sc F\"uredi, Jiang, Kostochka, Mubayi, and Verstra\"ete:   Splitting theorem
}}}

\title{
Partitioning ordered hypergraphs}

\author{
\hspace{0.8in} Zolt\'an F\" uredi\thanks{Research supported by grant KH130371
from the National Research, Development and Innovation Office NKFIH.}
\and
Tao Jiang\thanks{Research partially supported by National Science Foundation award DMS-1400249.}
\and
Alexandr Kostochka\thanks{Research  supported in part by NSF grant
 DMS-1600592 and by grants  18-01-00353A and 19-01-00682
  of the Russian Foundation for Basic Research.
} \hspace{0.8in} \smallskip \and
Dhruv Mubayi\thanks{Research partially supported by NSF awards DMS-1300138 and 1763317.} \and Jacques Verstra\"ete\thanks{Research supported by NSF awards DMS-1556524
and DMS-1800332.}
}

\maketitle

\begin{abstract}

An {\em ordered $r$-graph} is an $r$-uniform hypergraph whose vertex set is linearly ordered. Given $2\leq k\leq r$, an ordered $r$-graph $H$
is {\em interval} $k$-{\em partite}  if
there exist at least $k$ disjoint intervals in the ordering  such that every edge of $H$ has nonempty intersection with each of the intervals and is contained in their union.
% it is {\em interval} $r$-{\em partite} if there are $r$ consecutive intervals such that each edge has one vertex in each interval.
 
 Our main result implies that for each $\alpha > k - 1$ and $d>0$, every $n$-vertex  ordered $r$-graph  with $d \,n^{\alpha}$ edges has for some $m\leq n$ an $m$-vertex  interval $k$-partite subgraph  with $\Omega(d\, m^{\alpha})$ edges. This is an extension to ordered $r$-graphs of the observation
 by Erd\H os and Kleitman that every $r$-graph contains an $r$-partite subgraph with a constant proportion of the edges.
% We prove that each very dense ordered $r$-graph has   a relatively dense subgraph that has the following similar structure.
%For $2 \le k \le r$, an ordered $r$-graph  %For a non-negative real $\alpha$, the $\alpha$-{\em density} of an $n$-vertex $r$-graph  with $m$ edges is  $m/n^{\alpha}$.
 The restriction  $\alpha > k-1$ is sharp. 
We also present applications of the main result to several extremal problems for ordered hypergraphs.
\end{abstract}

\section{Introduction}

We let  $[n]= \{1, \ldots, n\}$ and use standard asymptotic notation; in particular, given functions $f, g: \mathbb{Z^+} \to \mathbb{R^+}$, we write $f(n) = \Omega(g(n))$
if there exists $c>0$ such that $f(n)\ge c g(n)$ for all $n \geq 1$.  We associate a hypergraph $H$ with its edge set and write $e(H)$ for the number of  the edges and $v(H)$ for the number of the vertices  in $H$.

\smallskip

 An $r$-{\em graph} is a hypergraph with all edges of size $r$; it is $r$-{\em partite} if there is a partition of the vertex set into $r$ parts such that every edge has exactly one vertex in each part. The following observation is due to Erd\H{o}s and Kleitman:

 \begin{comment}
  Every $r$-graph  contains an $r$-partite subgraph with at least $r!/r^r$ proportion of its edges. An important consequence of this basic observation  due to Erd\H os and Kleitman~\cite{EK} is that any extremal problem for $r$-graphs can be reduced to the corresponding extremal problem where the underlying $r$-graph is $r$-partite with the loss of only a constant multiplicative factor:
\end{comment}

{\bf Proposition A.} (Erd\H{o}s-Kleitman~\cite{EK}) {\em
Every $r$-graph contains an $r$-partite subgraph with at least $r!/r^r$ proportion of its edges.}

In particular, any extremal problem for $r$-graphs can be reduced to the corresponding extremal problem where the underlying $r$-graph is $r$-partite with the loss of only a constant multiplicative factor. In this paper, we consider analogs of this result in the ordered hypergraph setting and illustrate their use on some ordered extremal hypergraph problems.

\smallskip

 An {\em ordered hypergraph} is a hypergraph together with a linear ordering of its vertex set.
Extremal problems on ordered hypergraphs arose from several sources, in particular, from combinatorial geometry, enumeration of permutations with forbidden subpermutations, and the study of matrices  with forbidden submatrices --
see for instance Anstee~\cite{Anstee,AnFu}, F\"{u}redi and  Hajnal~\cite{FurHaj}, Pach and Tardos~\cite{PT}, Marcus and Tardos~\cite{MT}, Tardos~\cite{T}, Fox~\cite{Fox}.

\smallskip

 Let $V$ be a linearly ordered set.  An {\em interval} in $V$ is a set of consecutive elements in the ordering.
For $A,B\subset V$, we write $A<B$ to mean that $a<b$ for every $a \in A, b \in B$. A key definition in our work is the following:

\begin{definition}
Let $k$ be a positive integer.
An ordered $r$-graph $H$ is {\em interval} $k$-{\em partite} if for some $\ell\ge k$ there are intervals $I_1<I_2< \cdots <I_{\ell}$ such that every edge of $H$ is contained in $I_1\cup \ldots\cup I_\ell$ and has nonempty intersection with $I_j$ for each $1\le j \le \ell$.
\end{definition}

In particular, an ordered $r$-graph $H$ is {\em interval $r$-partite} if there exist  intervals
$I_1< I_2< \cdots< I_r$ in $V(G)$ such that every edge of $H$ contains exactly one vertex from each $I_i$.
In these terms, the Erd\H os--Kleitman observation, Proposition A, does not hold for ordered graphs as witnessed by the following simple example:  every interval bipartite subgraph of the ordered graph with vertex set $[2n]$ and
edge set $\{\{2i-1,2i\}: 1\leq i\leq n\}$ has at most one edge. However, Pach and Tardos~\cite{PT}
showed that dense ordered graphs contain relatively dense interval bipartite  graphs using the following result:

{\bf Theorem B.} (Pach and Tardos~\cite{PT}). {\em
	Each ordered  $n$-vertex graph $G$ is the union of edge-disjoint subgraphs $G_i$ for $0 \leq i \leq  \lfloor \log_2 n\rfloor$ such that
	each $G_i$ is a union of at most $2^i$ interval bipartite graphs with parts of size at most $\lceil n/2^i \rceil$.}

Our first main result is the following ordered hypergraph analog of Theorem B:

\begin{comment}
to show that for $\alpha> k-1$, every $n$-vertex $r$-graph
	 with $\alpha$-density $d$
contains an interval  $k$-partite  subgraph  with $\alpha$-density  $\Omega(d)$.

\begin{thm} \label{splitting}
Let $r-1 \ge k \ge 1$ and $k < \alpha \le r$ and let $H$ be an ordered $r$-graph with $d_{\alpha}(H) =d$. Then $H$ contains an interval $(k+1)$-partite  subgraph $H'$ with $d_{\alpha}(H') > c \cdot d$, where
$$c=c(\alpha, r, k) > \frac{1-k^{k-\alpha}}{r^{\alpha}k^{2r+\alpha}}.$$
\end{thm}
\end{comment}

\begin{thm} \label{splittinglem}
Let $2 \leq k \leq r \leq n$ be integers. Then every ordered $n$-vertex $r$-graph $H$ is the union of
edge-disjoint ordered $r$-graphs $H_i$ for $0 \leq i \leq \lfloor \log_2 n\rfloor$ such that each $H_i$ is a union of at most $\frac{1}{(k-1)!}\sum_{j = k}^r {2k-2 \choose j} \cdot 2^{i(k - 1)}$ interval $k$-partite $r$-graphs with parts of size at most $\lceil n/2^i\rceil$.
\end{thm}

For $k = r = 2$, Theorem \ref{splittinglem} corresponds to Theorem B. Note that Theorem B easily implies the following, which appears
implicitly in Pach and Tardos~\cite{PT}:

{\bf Theorem C.} {\em For each real $\alpha \geq 1$, $d > 0$ and $n>1$, if $G$ is an ordered $n$-vertex graph with $e(G) = dn^{\alpha}$,
then for some $m \in [n]$, $G$ contains  an interval bipartite subgraph $G'$ with parts of size at most $m$ and
\begin{equation}\label{graphsplit}
e(G') = \left\{\begin{array}{ll}
\Omega\Bigl(\displaystyle{\frac{dm^{\alpha}}{\log_2 n}}\Bigr) & \mbox{ if } \alpha = 1 \\ \\
\Omega(dm^{\alpha}) & \mbox{ if } \alpha > 1
\end{array}\right.
\end{equation}
}

As observed by Pach and Tardos~\cite{PT}, the logarithmic factor in (\ref{graphsplit}) for $\alpha = 1$ is necessary: for
the ordered path $P$ with edges $\{v_i,v_{i + 1}\} : 1 \leq i \leq 4$ such that $v_2 < v_4 < v_3 < v_1 < v_5$, extremal $n$-vertex ordered $P$-free
graphs have $n\log n + O(n)$ edges, whereas an extremal $n$-vertex interval bipartite $P$-free graph has $\Theta(n)$ edges (see F\"{u}redi~\cite{Furedi},
Bienstock and Gy\"{o}ri~\cite{BG}, and Tardos~\cite{Tardos}).

%\subsection{Main results}
%The goal of this paper is to extend Propositions A -- C to ordered $r$-graphs and to solve several extremal problems.
%It is convenient to introduce the following definition:
%\begin{definition}	For $\alpha\geq 1$, the {\em $\alpha$-density} of a hypergraph $H$ is
%$d_{\alpha}(H) = e(H)/v(H)^{\alpha}$.
%\end{definition}
%Theorem C states that every $n$-vertex ordered graph $G$ has an interval bipartite subgraph $G'$ with $d_{\alpha}(G') = \Omega(d_{\alpha}(G))$ if $\alpha > 1$, whereas
%$d_{\alpha}(G') = \Omega(d_{\alpha}(G)/\log n)$ if $\alpha = 1$.
Our second main result is the following generalization of Theorem C to ordered $r$-graphs:

\begin{thm} \label{splitting}
Let $2 \leq k \leq r \leq n$ be integers and let $\alpha$ be a real number with $k - 1 \leq \alpha \leq r$. Then every ordered $r$-graph $H$ with $n$ vertices and
$dn^{\alpha}$ edges has an interval $k$-partite subgraph $H'$ with parts of size at most $m$ for some $m \in [n]$ and
\begin{equation}\label{one}
 e(H') = \left\{ \begin{array}{ll}
\displaystyle{\Omega\Bigl(\frac{dm^{\alpha}}{\log_2 n}\Bigr)} & \mbox{ if }\alpha = k - 1\\ \\
\displaystyle{\Omega(dm^{\alpha})} & \mbox{ if }\alpha > k - 1
\end{array}\right.
\end{equation}
\end{thm}

The case $k = r = 2$  is Theorem C.

{\bf Remarks.}
\vspace{-5mm}
\begin{itemize}
	\itemsep0.3em

	\item Theorem~\ref{splitting} is sharp in that for $2 \leq k < r$ and $\alpha = k - 1$, there exist $n$-vertex $r$-graphs $H$ with $e(H)  = dn^{\alpha}$
where every interval $k$-partite  subgraph $H'$ with parts of size $m$ has $e(H') = O(dm^{\alpha}/\log n)$, and for $\alpha < k - 1$, there exist $n$-vertex $r$-graphs $H$ with $e(H) = dn^{\alpha}$
where every interval $k$-partite  subgraph $H'$ has $e(H') =  O(dn^{\alpha - a})$ where $a = \min\{1,k-1-\alpha\} > 0$. We will prove this in Section~\ref{construction} (see Constructions 1 and 2).

\item For $\alpha > k - 1$, Theorem \ref{splitting} guarantees that an $n$-vertex ordered $r$-graph with $\Theta(n^{\alpha})$ edges has an
interval $k$-partite subgraph parts of size $m$ and $\Theta(m^{\alpha})$ edges for some $m \in [n]$. In sharp
contrast with the Erd\H os--Kleitman Lemma, Proposition A, the value of $m$ may be necessarily be small relative to the number of
vertices in the host $r$-graph: we give a construction in Section \ref{construction} (see Construction 3) where we need $m = O(n^{1-1/\alpha})$ for $\alpha > k - 1$.

\item We do not optimize the constant $c = c(\alpha,k,r)$ in the bound $e(H') \geq cdm^{\alpha}$ for $\alpha > k - 1$ in Theorem \ref{splitting}.
The proof of Theorem \ref{splitting} gives
\begin{equation}\label{cvalue}
c(\alpha,k,r) \geq \frac{(k-1)!(1 - 2^{k-1-\alpha})}{\sum_{j = k}^r {2k - 2 \choose j}}.
\end{equation}
In particular, $c(r,r,r) \geq (r-1)!4^{-r}$, whereas
%it is straightforward to see that 
 for every $r$-partite subgraph $H'$ of $K_n^r$ with parts of size $m$,
$e(H') \leq m^r$, and so $c(r,r,r) \leq r!$.

\item For each partition $\pi$ of $r$, one can extend Theorem \ref{splitting} to the setting of {\em interval $\pi$-partite subgraphs} -- here $\pi$ specifies the number of vertices of an edge in each part -- by replacing the range of $\alpha$ to $\alpha\ge f(\pi)$ where $f(\pi)$ is the maximum length of a partition that is not a refinement of $\pi$. For example, if $\pi= 1+1+\cdots + 1$, then $f(\pi) = r-1$, if  $\pi= 2+1+\cdots + 1$, then $f(\pi) = r-2$ and if $\pi = 1+(r-1)$, then $f(\pi) = \lfloor r/2 \rfloor$. This has other interesting consequences which we will explore in forthcoming work.

\end{itemize}

\subsection{Applications of Theorem~\ref{splitting}}

We next describe how to apply Theorem \ref{splitting} to a variety of ordered extremal problems and convex geometric extremal problems for families of $r$-graphs. This enables us to transfer  classical extremal problems  to the ordered setting via Theorem \ref{splitting}. The following definition is needed:

\begin{definition} For an $r$-partite $r$-graph $F$, $\ord(F)$ denotes the family of interval $r$-partite $r$-graphs
 isomorphic to $F$.
For a family $\mathcal{F}$ of $r$-partite $r$-graphs,  $\ord(\mathcal{F}) =\bigcup_{F \in \mathcal{F}}\ord(F)$.
% is the family of interval $r$-partite $r$-graphs  isomorphic to some $F \in \mathcal{F}$.
\end{definition}

Note that $\ord(\mathcal{F})$ may be empty but not in the cases we investigate. A first and natural example is the case that $\mathcal{F}$ consists of the $r$-graph of two disjoint edges. The Erd\H{o}s-Ko-Rado Theorem~\cite{EKR} states that for $n \geq 2r + 1$, the unique extremal $n$-vertex $r$-graph without two disjoint edges consists of
all $r$-sets containing one vertex, with ${n - 1 \choose r - 1}$ edges. In~\cite{FJKMV2}, the following ordered version of the Erd\H{o}s-Ko-Rado Theorem is proved:

\begin{thm} {\bf (\cite{FJKMV2})} \label{ekr}
Let $r \geq 3$ and $n \geq 2r + 1$. Then the maximum number of edges in an ordered $n$-vertex $r$-graph that does not contain
two edges of the form $\{v_1,v_2,\dots,v_r\}$ and $\{w_1,w_2,\dots,w_r\}$ such that $v_1 < w_1 < v_2 < w_2 < \dots < v_r < w_r$
is exactly ${n \choose r} - {n - r \choose r}$.
\end{thm}
For an ordered $r$-graph $F$, let $\ex_{\to}(n, F)$  denote the maximum number of edges in an $n$-vertex ordered $r$-graph  that does not contain $F$.
For a family  $\mathcal{F}$ of  ordered $r$-graphs,  let $\ex_{\to}(n, \mathcal{F})=\min_{F\in \mathcal{F}}\ex_{\to}(n, F)$.
In this language Theorem~\ref{ekr} implies that  for $n \geq 2r + 1$,
\[ \ex_{\to}(n,\ord(F)) \le {n \choose r} - {n - r \choose r},\]
 where $F$ is the $r$-graph comprising two disjoint edges (in fact, it applies to a particular member of $\ord(F)$).
Results for hypergraph matchings (i.e., for sets of disjoint edges) by Klazar and Marcus~\cite{KM} show that for each interval $r$-partite matching $M$,
$\ex_{\to}(n,M) = O(n^{r - 1})$, thereby extending the celebrated Marcus-Tardos~\cite{MT} theorem for matchings in ordered graphs to ordered $r$-graphs. We now give some further examples where classical extremal problems are transferred to the ordered setting via Theorem \ref{splitting}.

\medskip

\subsubsection{ Simplices} 
A {\em $d$-dimensional $r$-simplex} is an $r$-graph of $d + 1$ edges such that any $d$ of the edges have non-empty intersection,
but all $d + 1$ edges have empty intersection.
Denote by $\mathcal{S}_{d}^r$ the family
of $d$-dimensional $r$-simplices. The set  $\mathcal{S}_{d}^r$ is non-empty if $r\geq d$.
The study of these abstract simplices in the context of extremal hypergraph theory was first initiated by Chv\'atal who posed the following conjecture.

\begin{conjecture} {\bf (Chv\'atal~\cite{CH})} \label{chv}
Let $r\ge d+1\ge 3$ and $n \ge r(d+1)/d$. Then $\ex(n,  \mathcal{S}_{d}^r) = {n-1 \choose r-1}$.
\end{conjecture}

Frankl and F\"{u}redi~\cite{FF} proved Conjecture~\ref{chv} for large $n$ (Keller and Lifschitz~\cite{KL} improved the bounds on $n$) and Mubayi and Verstra\"ete~\cite{MV2005} proved it for $d=2$,
which was a problem of Erd\H os. Very recently, Currier~\cite{Currier} proved the conjecture for $n \geq 2r$. We prove the following theorem.

\begin{thm}\label{simplex}
For all fixed $r \ge  d+1 \geq 3$,
\[ \ex_{\to}(n,\ord(\mathcal{S}_{d}^r)) = \Theta(n^{r-1}).\]
\end{thm}

\subsubsection{Expansions} Our next example is more general.
If $\mathcal{F}$ is a family of $(r - 1)$-graphs, let $\mathcal{F}^+$ denote the family of $r$-graphs $F^+$ obtained from each $F \in \mathcal{F}$ by adding a vertex $v_e$ to edge $e \in F$ such that all the vertices $v_e : e \in F$ are distinct from each other
 and from the vertices of $F$. A study of extremal problems for families $\mathcal{F}^+$ is given in~\cite{MV},
where $F^+$ is referred to as an {\em expansion} of $F$. Such families lend themselves naturally to an application of Theorem \ref{splitting}:

\begin{thm}\label{expansion}
	Let $r \ge 3$ and $\mathcal{F}$ be a family of $(r - 1)$-graphs with $\ex_{\to}(n,\ord(\mathcal{F})) = O(n^{r-2})$.
	Then
	\[ \ex_{\to}(n,\ord(\mathcal{F}^+)) = O(n^{r-1}).\]
\end{thm}
Actually, our proof yields a stronger fact. Recall that for a vertex $v$ in an $r$-graph $G$, the {\em link of $v$ in $G$}
is the $(r-1)$-graph $G(v)$ whose edge set  is $\{A-v;\, v\in A\in E(G)\}$. Our proof shows that for some $C$
each ordered $n$-vertex $r$-graph with at least $Cn^{r-1}$ edges contains an $r$-graph obtained from an $F_0\in \ord(\mathcal{F}^+)$ by
choosing some vertex $v_0$ of degree $1$ in $F_0$ and
adding all the edges of the kind $L\cup {v_0}$ where  $L$ is an $(r-1)$-subset of an edge in $F_0$.
% one of the vertices $v_e$ contains all $(r-1)$-sets of $F \in {\cal F}$ in its link.
 This stronger fact implies that $$\ex_{\to}(n, \ord(T_r)) = O(n^{r-1}),$$ where $T_r = \{e,f,g\}$ is the loose $r$-uniform triangle i.e. $|e \cap f| = |f \cap g| = |g \cap e| = 1$ and $e \cap f \cap g =\emptyset$.

%\begin{thm} \label{forest} Let $F$ be a forest. Then $\ex_{\to}(n, \ord({F})) = O(n)$.
%	\end{thm}

\bigskip

\subsubsection{ Hypergraph forests}
Our next application concerns hypergraph forests.
The {\em shadow} $\partial H$ of an $r$-graph $H$ is the collection of $(r-1)$-sets contained in some edge of $H$.
We follow Frankl and F\"{u}redi~\cite{FF} for an inductive definition
of trees in hypergraphs: a single edge is a tree, and given any tree $T$ with edges $e_1,e_2,\dots,e_h$, a tree with $h + 1$ edges is obtained by selecting
$f \in \partial T$ and a vertex $x$ not in $T$, and adding the edge $f \cup \{x\}$. A {\em forest} is a subgraph of a tree.
By definition, each $2$-uniform tree (respectively, $2$-uniform forest) is a tree (respectively, forest) in the usual sense.
%Extending Theorem \ref{forest}, and u
 Using Theorem \ref{splitting}, we prove the following:

\begin{thm} \label{r-forest} Fix $ r \ge 2$ and let $F$ be an $r$-uniform forest. Then $\ex_{\to}(n, \ord({F})) = O(n^{r-1})$.
\end{thm}

{\bf Remarks.}
\vspace{-5mm}
\begin{itemize}
	\itemsep0.3em
	\item A conjecture of Pach and Tardos~\cite{PT} would imply $\ex_{\to}(n,T) = n^{1 + o(1)}$ for every $2$-interval-partite tree $T$
	with at least two edges. Theorems \ref{expansion} and \ref{r-forest} suggest that perhaps
	for every interval $r$-partite $r$-uniform tree $T$, $\ex_{\to}(n,T) \leq n^{r - 1 + o(1)}$.
	\item It remains an intriguing open problem to determine for which $r$-graph families $\cal F$  \begin{equation} \label{fails} \ex(n, {\cal F}) = O(n^{r-1}) \qquad \Longrightarrow \qquad  \ex_{\to}(n, \ord({\cal F})) = O(n^{r-1}).\end{equation}
	According to Theorem \ref{r-forest}, this is true for $r = 2$. Since for every $r$-uniform forest $F$, $\ex(n,F) = O(n^{r - 1})$,   Theorem \ref{r-forest} yields that
	the above implication is also true if $\mathcal{F}$ contains an $r$-uniform forest. We do not know any explicit  example for $r \geq 3$ for which (\ref{fails}) fails, although we believe that many such examples exist.
	
	\item In~\cite{FJKMV}, we
	heavily used the $k=r-1$ case of Theorem~\ref{splitting} to prove that the extremal function of so called {\em crossing paths} in
	convex geometric hypergraphs  has order $n^{r-1}$ or $n^{r-1} \log n$.
	%They were our main tools to prove upper bounds when
	%the length of a path was at least $r+2$.
	
\end{itemize}

\subsubsection{Ordered Ruzsa-Szemer\'{e}di Theorem} 
We consider the ordered version of the famous Ruzsa-Szemer\'edi $(6,3)$-Theorem~\cite{RS} which states that the maximum number of edges in an $n$-vertex 3-graph with no 6 vertices spanning 3 edges is $o(n^2)$. This is equivalent to  the statement $\ex(n, {\cal F}_{RS}) = o(n^2)$
where ${\cal F}_{RS} = \{I_2, T_3\}$ and $I_2$ is the 3-graph comprising two edges sharing exactly two points.

\begin{thm} \label{thmrsz} Let ${\cal F}_{RS} = \{I_2, T_3\}$. Then  $\ex_{\to}(n, \ord({\cal F}_{RS})) = o(n^{2})$.
\end{thm}

\subsubsection{Forbidden ordered intersections}
Our final example addresses an $r$-graph  problem whose answer has order of magnitude $n^{\alpha}$ where $\alpha \ne r-1$. Let $I^r(\ell)$ denote the $r$-graph consisting of two edges sharing exactly $\ell$ vertices.
The study of  $\ex(n, I^r(\ell))$ was initiated by Erd\H os. Frankl and F\"uredi~\cite{FF} proved that
\begin{equation}
\ex(n, I^r(\ell))= \Theta(n^{\max \{\ell, r-\ell-1\}}) \qquad \hbox{ \rm{for $0\leq \ell\leq r-1$}}.
\end{equation}
We are able to extend this result to the ordered setting using Theorem \ref{splitting}:

\begin{thm} \label{exactell} For $r \ge 2$ and $1 \leq \ell \leq r-1$, and $\alpha = \max\{\ell, r - \lceil(\ell + 1)/2\rceil\}$
	$$ \Omega(n^{\alpha}) = \ex_{\to}(n, \ord(I^r(\ell))) =
	\left\{ \begin{array}{ll}
	O(n^{\alpha}) & \mbox{ if } \ell \mbox{ is odd} \\
	O(n^{\alpha}\log n) & \mbox{ if } \ell \mbox{ is even}.
	\end{array}\right.$$
\end{thm}

Note that the $\ell=0$ case is covered by Theorem~\ref{ekr} which gives
$\ex_{\to}(n, \ord(I^r(0))) =\Theta(n^{r-1})$. A construction of a dense $\mbox{ord}(I^r(\ell))$-free ordered $r$-graph is given in Construction 4. We believe the $\log n$ factor when $\ell$ is even
can be removed, so that $\ex_{\to}(n, \ord(I^r(\ell))) = \Theta(n^{\alpha})$.

%For example, in~\cite{FJKMV2}, the authors proved that if $P_k^r$ is an $r$-uniform tight path with $k$ edges, and $k>r+1$, then for an appropriate ordered $r$-graph $CP_k^r$ isomorphic to $P_k^r$ with interval chromatic number $r$, we have $\ex_{\to}(n, CP_k^r) = \Omega(n^{r-1} \log n)$, while it is easy to see that in the unordered case $\ex(n, P_k^r) = O(n^{r-1})$. But this does not prove that  $\ex_{\to}(n, \ord( P_k^r)) = \Omega(n^{r-1} \log n)$, since there is a different ordered $r$-graph $ZP_k^r$ isomorphic to $P_k^r$ with interval chromatic number $r$ for which $\ex_{\to}(n, ZP_k^r) = O(n^{r-1})$; hence $\ex_{\to}(n, \ord(P_k^r)) = O(n^{r-1})$.

\medskip

We present four constructions  in Section \ref{construction} and prove
Theorem \ref{splitting}  in Section~\ref{split1proof}.  Theorems~\ref{simplex}--\ref{exactell} are proved in
Section \ref{app-proof}.

\section{Constructions}\label{construction}

%In this section, we give two constructions.
 % Our first construction shows that Theorem B fails for $r$-graphs when $r \ge 3$.
%More generally, for each $n>>r\geq 3$  we construct  an $r n$-vertex $r$-graph  $H^r_n$ such that %for each   $1\leq \alpha\leq r-2$,
% We present it in in the more general language of $\alpha$-densities.

Our first construction requires the following lemma.

\begin{lem}\label{leM} Let $H_{n_1,n_2}$ be the ordered bipartite graph  with vertex set $[n_1+n_2]$ and
	parts $A=[n_1]$ and $B=\{n_1+1,\ldots,n_1+n_2\}$ such that for $i<j$, the pair $ij$ is an edge in $H(n_1,n_2)$ iff $1\leq i\leq n_1<j\leq n_1+n_2$ and
	$j-i$ is a power of $2$. Then for each $A'\subseteq A$ and $B'\subseteq B$, the number of edges in $H_{n_1,n_2}[A'\cup B']$
	is at most $2|A'\cup B'|$.
\end{lem}

\begin{proof} Suppose for some $A'\subseteq A$ and $B'\subseteq B$, graph $H:=H_{n_1,n_2}[A'\cup B']$ has  more than $2(|A'| + |B'|)$ edges. Let us assume
	$A' = \{a_1,a_2,\dots,a_l\}$ and $B' = \{b_1,b_2,\dots,b_m\}$ where
	\[ a_1 < a_2 < \dots < a_l < b_1 < b_2 < \dots < b_m.\]
	For each vertex $a\in A'$, remove from $H$ the edges $\{a,b_i\}$ and $\{a,b_j\}$ where $i$ is minimum index for which such an edge exists, and $j$ is maximum index for which such an edge exists. Repeat this procedure for $b \in B'$ with respect to vertices in $A'$. Since $H$ has more than $2(|A'| + |B'|)$ edges, and we removed at most $2(|A'| + |B'|)$ edges, the remaining graph $H'$ has
	an edge $\{a,b\}$ with $a \in A', b \in B'$. Now there exist vertices $a'$ and $b'$ such that $\{a',b\}$ and $\{a,b'\}$ are edges
	and $a < a' < b' < b$. However, it is not possible for $b - a$, $b' - a$ and $b - a'$ all to be powers of $2$.
\end{proof}

\medskip

Our first construction shows the logarithmic factor in the first bound in Theorem \ref{splitting} is necessary.

\medskip

{\bf Construction 1:} {\em An $n$-vertex $r$-graph $H_r(n)$ with $dn^{k - 1}$ edges such that for every $m\leq n$
\begin{equation}
e(H') = O\Bigl(\frac{dm^{k - 1}}{\log n}\Bigr)
\end{equation}
for every interval $k$-partite $H' \subset H_r(n)$ with parts of size $m$.}

\smallskip

The vertex set of $H_r(n)$ is $[n]$  ordered as
$1 < 2 < \dots < n$. The edges of $H_r(n)$ are the sets $\{v_1,v_2,\dots,v_r\}$
with $v_1 < v_2 < \dots < v_r$ such that the difference $v_{i+1} - v_{ i}$ is a power of $2$ for $i=1, \ldots, r-k+1$.
Then $e(H_r(n)) = \Theta(n^{k-1}(\log n)^{r-k+1}) = \Theta(dn^{k - 1})$ where $d = (\log n)^{r - k + 1}$. Let $H'$ be any interval $k$-partite subgraph of $H_r(n)$, with ordered parts $I_1<I_2<\ldots <I_k$ each of size $m$,
  and $G'$ be the bipartite graph whose edges are pairs
$\{v,w\} \subset e = \{v_1,v_2,\dots,v_r\} \in H'$ where $v$ is the largest vertex in $I_1$ and $w$ is the smallest vertex in $I_2$. Note, crucially,  that $w-v$ must be a power of 2, since otherwise the $r-k+2$ smallest vertices of $e$ lie in $I_1$,  which means that some $I_j$ is empty.
Lemma~\ref{leM} now yields that $e(G')\leq 2v(G') \leq 4m$. But then
$e(H') = O(m^{k-1} (\log n)^{r-k})$, so $e(H') = O(dm^{k-1}/\log n)$.

\medskip

 Our next construction shows that the bound $\alpha \geq k - 1$ in Theorem~\ref{splitting} cannot be improved.

{\bf Construction 2:} {\em An $r n$-vertex $r$-graph  $H^r_n(k)$ with $dn^{\alpha}$ edges such that for $1 \leq \alpha < k - 1 \leq r - 1$ and $a = \min\{1,k-1-\alpha\} > 0$,
\begin{equation}\label{hrn}
e(H') = O(dm^{\alpha}/n^{a})
\end{equation}
for every interval $k$-partite $H'\subset H^r_n(k)$ with parts of size $m$.}

\smallskip

For $k \leq r$ and $1 \leq j \leq n$, let $I_j = \{(r-k+2)j-(r-k+1),(r-k+2)j-(r-k+1)+1,\dots,(r-k+2)j\}$, so that $|I_j| = r - k + 2$, and let
$H^r_n(k)$ be the ordered $r$-graph
with vertex set $[r n]$ and edge set
$$\{I_j \cup \{a_{r-k+3},a_{r-k+4},\ldots,a_r\}\,: 1 \leq j\leq n,\;\mbox{and}\; (\ell-1)n<a_{\ell}\leq \ell n\;\mbox{for} \; r-k+3\leq \ell\leq r\}.$$
By definition, $e(H^r_n(k)) = n^{k-1} = dn^{\alpha}$ where $d = n^{k-1-\alpha}$.
On the other hand, let $H'$ be  an interval $k$-partite subgraph $H'$ of $H^r_n(k)$ with parts of size $m$. Then
  there exists $j$ such that every edge of $H'$ contains $I_j$. In particular,
  \[ e(H') = O(m^{k-2}) = O\Bigl(dm^{\alpha} \cdot \frac{m^{k - 2 - \alpha}}{n^{k - 1 - \alpha}}\Bigr).\]
  If $\alpha \leq k - 2$, then $m^{k - 2 - \alpha}/n^{k - 1 - \alpha} = O(1/n)$. If $\alpha > k - 2$,
  then $m^{k - 2 - \alpha}/n^{k - 1 - \alpha} = O(n^{k - 1 - \alpha})$,  as claimed in~\eqref{hrn}.

%a decomposition with properties as in Theorem B
%would need a linear in $n$ number of subgraphs. In particular (for fixed $\alpha$ and $n \rightarrow \infty$), $H^r_n$ has the property that

%{\color{red} ! ! ! XXX ???
%p.1 after Construction 1Lemma A does not cliam to decompose a graph into $\log n$ connected pieces.
%(It is not possible). But each factor's each connected component is bipartite.
%This example is OK, but we have to say correctly (i.e, in each factor the UNION of connected r-int colored component constitutes only $O(1/n)$ part of the edges).

%I am not sure how to answer ZF comment above ..}

\medskip

The next construction shows that the interval $k$-partite subgraph guaranteed by Theorem \ref{splitting} may have few vertices.

\medskip

{\bf Construction 3:} {\em Let $k - 1 < \alpha \leq r$. We give an ordered $n$-vertex $r$-graph $H(n,r)$ with $dn^{\alpha}$ edges
such that for every interval $k$-partite subgraph $H'$ of $H(n,r)$ with parts of size $m$ and
$e(H') = \Omega(dm^{\alpha})$,
\[ m = O(n^{1 - 1/\alpha}).\]}

Consider the ordered $r$-graph $H = H(n,r)$ with vertex set $[n]$ and edge set $\{[i,i+r-1] : 1 \leq i \leq n - r + 1\}$.
Then
\[ e(H)  = n - r + 1 = dn^{\alpha}\]
where $d = \Theta(n^{1 - \alpha})$.
On the other hand, if $H'$ is an interval $k$-partite subgraph with parts of size $m$, then $H'$ cannot contain two
disjoint edges, so $e(H') \leq r$. So if $e(H') = \Omega(dm^{\alpha})$, then
$dm^{\alpha} = O(r)$ so $m^{\alpha} = O(1/d) = O(n^{\alpha - 1})$.

\medskip

Our final construction provides a lower bound on $\ex_{\to}(n,\mbox{ord}(I^r(\ell))$ for Theorem \ref{exactell}:

\medskip

{\bf Construction 4:} {\em For $0 \leq \ell \leq r - 1$ and $\alpha = \max\{\ell,r - \lceil (\ell + 1)/2 \rceil\}$,
we give an ordered $6n$-vertex $r$-graph $H(n,r,\ell)$ with $\Omega(n^{\alpha})$ edges not containing
any member of $\mbox{ord}(I^r(\ell))$.}

For $\alpha = \ell$, let $H(n,r,\ell)$ be a Steiner $(n,r,\ell)$-system with any ordering of the vertices. Since 
$H(n,r,\ell)$ is $I^r(\ell)$-free, it is also $\mbox{ord}(I^r(\ell))$-free. If $\alpha = r - \lceil (\ell + 1)/2\rceil$, 
let $H(n,r,\ell)$ be defined as follows. The vertex set of $H(n,r,\ell)$ is $[6n]$. 
Let $M_2$ be the set of pairs $\{2i-1,2i\}$ for $1 \leq i \leq n$, and let $M_3$ be the set
of triples $\{2n + 3i-2,2n+3i-1,2n+3i\}$ for $1 \leq i \leq n$. If $\ell$ is odd, then the edges of $H(n,r,\ell)$ consist of $r - \ell - 1$ 
vertices from $[5n,6n]$, and $(\ell + 1)/2$ pairs from $M_2$. If $\ell$ is even, each edge of $H(n,r,\ell)$ 
consists of one triple from $M_3$, $\ell/2 - 1$ pairs from $M_2$, and $r - \ell - 1$ vertices from $[5n,6n]$. If 
$\ell$ is odd, then 
\[ e(H(n,r,\ell)) = {n \choose (\ell + 1)/2} \cdot {n \choose r - \ell - 1} = \Theta(n^{\alpha}).\]
If $\ell$ is even, then 
\[ e(H(n,r,\ell)) = n \cdot {n \choose \ell/2 - 1} \cdot {n \choose r - \ell - 1} = \Theta(n^{\alpha}).\]
Furthermore, $H(n,r,\ell)$ contains no member of $\mbox{ord}(I^r(\ell))$: if $\{e,f\} \in \mbox{ord}(I^r(\ell))$ and 
$e,f \in H(n,r,\ell)$, then $e \cap [5n] = f \cap [5n]$ and so $|e \cap f| = \ell + 1$, a contradiction.

\section{Proof of Theorems~\ref{splittinglem} and~\ref{splitting}}\label{split1proof}

{\bf Proof of Theorem \ref{splittinglem}.} Let $g = \lfloor\log_2 n\rfloor$, so that $2^{g} \leq  n < 2^{g+1}$.
	For $0 \leq i \leq g$, let $\mathcal{I}_i$ be the partition of $V(H)$ into intervals of length $2^{g-i}$ plus one interval of length at most $2^{g-i}$ containing the vertex $n$. Note $\mathcal{I}_g$ is the partition into singletons, so for each $e \in H$, there exists a minimum $i(e)$ such that $e$ intersects at least $k$ intervals in $\mathcal{I}_{i(e)}$. For $0 \leq i \leq g$, let
\[ H_i = \{e \in H : i(e) = i\}\]
 so that $H = \bigsqcup_{i=0}^g H_i$ -- the $H_i$ are edge-disjoint. Theorem \ref{splittinglem} follows from the following claim: for $0 \leq i \leq g$, $H_i$ is a union of $t_i \leq \sum_{j = k}^r {2k-2 \choose j} \cdot 2^{i(k - 1)}$ interval $k$-partite hypergraphs $H_{ij} : 1 \leq j \leq t_i$ with parts from $\mathcal{I}_i$ -- note that the parts in $\mathcal{I}_i$ each have size at most $2^{g - i} \leq  \lceil n/2^i\rceil$.
The claim is trivial for $i = 0$, since $H_0$ is empty unless $k = 2$ and $n > 2^g$, in which case $H_0$ is $k$-partite and $t_0 = 1$.
 To see the claim for $i \geq 1$, note that for $e \in H_i$, there are  $s\leq k - 1$ intervals $I_1,I_2,\dots,I_{s} \in \mathcal{I}_{i-1}$ such that
$e = \bigcup_{j = 1}^{s} (e \cap I_j) \subset I = \bigcup_{j = 1}^{s} I_j$,
  by the definition of $i(e) = i$. 
   If $|\mathcal{I}_{i-1}|\geq k-1$,
  then by adding intervals if needed we can assume $s=k-1$. In this case, there are at most ${|\mathcal{I}_{i-1}|\choose k-1}\leq \frac{|\mathcal{I}_{i-1}|^{k - 1}}{(k-1)!} \leq 
  \frac{2^{i(k - 1)}}{(k-1)!}$ choices for these $k-1$ intervals, and then at most $\sum_{j = k}^r {2k-2 \choose j}$ choices for the intervals from $\mathcal{I}_i$ contained in $I$ and intersecting $e$. 
  
  If $|\mathcal{I}_{i-1}|\leq k-2$,
  then again by adding intervals if needed we can assume $s=|\mathcal{I}_{i-1}|$, $I=[n]$, and have at most $\sum_{j = k}^r {2s \choose j}\leq \sum_{j = k}^r {2k-4 \choose j}$ choices for the intervals from $\mathcal{I}_i$ contained in $I$ and intersecting $e$. Since  $e$ intersects at least $k$ intervals in $\mathcal{I}_{i(e)}$,
   $k\leq 2s\leq 2^{i+1}$, and $2^{i(k - 1)}\geq \left(\frac{k}{2}\right)^{k-1}\geq 
  (k-1)!$; so the bound of the theorem holds again.
  \qed
	
	\medskip
	
{\bf Proof of Theorem \ref{splitting}.} We derive Theorem \ref{splitting} from Theorem \ref{splittinglem}, using the notation of its proof. Let $H$
be an $n$-vertex ordered $r$-graph with $dn^{\alpha}$ edges. Let $C = \sum_{j = k}^r {2k - 2 \choose j}$ and $m_i = 2^{g - i}$.
We prove that some $H_{ij}$ has
\[ e(H_{ij}) \geq \left\{\begin{array}{ll}
\displaystyle{\frac{1}{C}} \cdot  \frac{dm_i^{\alpha}}{1+\log_2 n} & \mbox{ if } \alpha = k - 1 \\ \\
\displaystyle{\frac{1 - 2^{k-1-\alpha}}{C}} \cdot d m_i^{\alpha} & \mbox{ if } \alpha > k - 1
\end{array}\right.
\]
Note that $H_{ij}$ is $k$-partite and the parts of $H_{ij}$ have size at most $m_i$, so the above statements imply Theorem \ref{splitting}.
Suppose, for a contradiction, that no $H_{ij}$ satisfies the above bounds.

{\bf Case 1.} $\alpha = k - 1$.  Then recalling $2^{g} \leq n$ and $t_i \leq C \cdot 2^{i(k - 1)}$,
\begin{eqnarray*}
e(H) \; \; = \; \; \sum_{i = 0}^g e(H_i) &\leq& \sum_{i = 0}^g \sum_{j = 1}^{t_i} e(H_{ij}) \\
&<& \sum_{i = 0}^g \sum_{j = 1}^{t_i} \frac{1}{C} \cdot  \frac{dm_i^{\alpha}}{1 + \log_2 n} \\
&=& \sum_{i = 0}^g \sum_{j = 1}^{t_i} \frac{d \cdot 2^{(k - 1)(g - i)}}{C(1 + \log_2 n)} \\
&\leq& dn^{k - 1} \sum_{i = 0}^g \frac{1}{1 + \log_2 n}   \; \; = \; \; (g + 1) \cdot  \frac{d \cdot n^{k - 1}}{1 + \log_2 n}.
\end{eqnarray*}
Since $g \leq \log_2 n$,  $e(H) <  d \cdot n^{k-1}$, a contradiction.

{\bf Case 2.} $\alpha > k - 1$. Let $c = (1 - 2^{k-1-\alpha})/C$. Then using $t_i \leq C \cdot 2^{i(k - 1)}$,
\begin{eqnarray*}
e(H) &\leq& \sum_{i = 0}^g \sum_{j = 1}^{t_i} e(H_{ij}) \\
&<& \sum_{i = 0}^g \sum_{j = 1}^{t_i} c \cdot dm_i^{\alpha}  \\
&\leq& \sum_{i = 0}^g C \cdot 2^{i(k - 1)} \cdot c \cdot d2^{\alpha(g - i)} \\
&\leq& dn^{\alpha} \cdot (1 - 2^{k-1-\alpha}) \cdot \sum_{i = 0}^g 2^{i(k-1-\alpha)}.
\end{eqnarray*}
Since $\alpha > k - 1$, the geometric series sum is less than $1/(1 - 2^{k-1-\alpha})$, and $e(H) < dn^{\alpha}$.
This contradiction completes the proof. \qed

\section{Proofs of Theorems \ref{simplex} -- \ref{exactell}}\label{app-proof}

Let $P_k^r$ denote the $r$-uniform tight path, which has vertex set $V=\{v_0,\ldots,v_{k+r-2}\}$ and edge set $\{\{v_i, v_{i + 1},\dots,v_{i + r - 1}\}:\, 0 \leq i \leq k-1\}$. Then $\ord(P_k^r)$ contains the ordered $r$-graph $ZP_k^r$ with edges $\{v_i, v_{i + 1},\dots,v_{i + r - 1}\}$ for $0 \leq i < k$
with a partition of $V$ into $r$ intervals $X_0 < X_1 < \dots < X_{r-1}$ such that vertices $v_i < v_{i + r} < v_{i + 2r} < \dots$ are in $X_i$ if $i$ is even and
$v_i > v_{i + r} > v_{i + 2r} > \dots$ in $X_i$ if $i$ is odd. Extremal problems for $ZP_k^r$ are studied in~\cite{FJKMV}, where the following theorem is (implicitly) proved:

\begin{thm}\label{zigzag}
For $k,r \geq 2$,
\[ \ex_{\to}(n,ZP_k^r) \leq (k - 1){n \choose r - 1}.\]
\end{thm}

In particular, this theorem gives the same upper bounds for the extremal function for $\ord(P_k^r)$, as $ZP_k^r \in \ord(P_k^r) $.   In~\cite{FJKMV, FJKMV2} we also obtain  ordered versions of the Erd\H{o}s-Ko-Rado Theorem  by taking every $r$th edge of $P_k^r$.

An ordered $r$-graph $H$ with vertex set $V$ is a {\em $(1,r-1)$-graph} if there exists an interval $X\subset V$ such that every edge of $H$ has exactly one vertex in $X$. Finally, an interval $(r-1)$-partite $r$-graph is a $(1, r-1)$-graph simply by combining parts.

\medskip

{\bf Proof of Theorem \ref{simplex}.} A {\em
strong $d$-dimensional $r$-simplex} $\hat{\mathcal{S}^r_d}$ is an $r$-graph consisting of $d + 2$ edges such that we may order the edges so that the first $d + 1$ edges form a $d$-simplex (see the definition in Section 1.3),
and the last edge contains at least one vertex from the intersection of every $d$-tuple of the edges of the $d$-simplex. For example, a strong 1-simplex comprises three edges $e,f,g$ such that  $e \cap f =\emptyset$ (so $e$ and $f$ form a 1-simplex), and both $e \cap g$ and $f \cap g$ are nonempty.
It is convenient to assume such an ordering of the edges of a strong simplex is given. We introduce strong simplices for the purpose of doing a
simple induction on $d$: we show that
\[ \ex_{\to}(n,\ord(\hat{\mathcal{S}^r_d})) \leq r^{10dr} n^{r - 1}.\]

The base case $d = 1$ follows easily from Theorem \ref{zigzag}: if $H$ is an ordered $r$-graph with more than
$r^{10r}{n \choose r - 1}$ edges, then $ZP_{r+1}^r \subset H$, and any three edges of $ZP_{r+1}^r$ that include the first and last edge form a strong 1-simplex.  Now suppose we have proved the
theorem for strong $d - 1$ simplices for some $d \geq 2$, and let $H$ be an $n$-vertex ordered $r$-graph with more than $ r^{10dr} n^{r - 1}$ edges. Applying Theorem~\ref{splitting} with $k = \alpha = r-1$
we find an interval $(r - 1)$-partite subgraph $G$ of $H$ with parts of size at most $m$ and
$$e(G) \geq c(r-1,r-1,r) \cdot r^{10dr} m^{r - 1}$$
for some $m > 0$ with intervals $X$ and $Y = V(G) - X$ as the parts of $G$. By (\ref{cvalue}), it is straightforward to check $c(r-1,r-1,r) > r^{-4r}$, and therefore
\begin{equation}\label{ebound}
 e(G) > r^{-4r}r^{10dr} m^{r - 1} > 2 r^{10(d-1)r} m^{r-1}.
 \end{equation}
We remove from $G$ each edge containing an $(r - 1)$-set in $Y$ which is contained in at most two edges of $G$. This way, we delete at most $2{m\choose r-1}$ edges. Since $r\geq 3$, this is less than $m^{r-1}$, and hence the remaining
  $(1,r-1)$-subgraph $G'$ of $G$ has at least $r^{10(d-1)r} m^{r - 1}$ edges. By averaging, some vertex $x \in X$ is contained in at least
  $$r^{1+10(d-1)r} m^{r - 2}\geq r^{10(d-1)(r-1)} m^{r - 2}$$ edges of $G'$. By induction, the link  of $x$ in $G'$ contains a strong $(d - 1)$-dimensional simplex $F$, say with edges $e_1,e_2,\dots,e_d,f$, with $e_1,e_2,\dots,e_d$ forming a
$(d - 1)$-dimensional simplex. Since $f$ is contained in at least $3$ edges of $G$, there exists $y \neq x$ such that $f \cup \{y\} \in G$. Then $e_1 \cup \{x\},e_2 \cup \{x\},\dots,e_d \cup \{x\},f \cup \{y\}$
is a $d$-dimensional simplex in $H$, and together with $f \cup \{x\}$, we have a strong simplex in $H$. This proves the theorem. \qed

\medskip

{\bf Proof of Theorem \ref{expansion}.}
Let $M$ denote the largest number of edges in an $r$-graph in $\mathcal{F}^+$ and suppose $\ex_{\to}(n,\ord(\mathcal{F})) \leq cn^{r - 2}$ for all $n>1$. We will prove that $\ex_{\to}(n,\ord(\mathcal{F}^+)) \le c' n^{r-1}$ where $c' = \left(\frac{M+c}{2}\right) r^{10r}$. Suppose that $H$ is an $n$-vertex $r$-graph
with more than $c'n^{r-1}$ edges.
Applying Theorem~\ref{splitting} with $k=\alpha = r-1$,
we find an $m$-vertex  $(1,r-1)$-subgraph $G$ of $H$ with at least
$r^{-4r}c' m^{r - 1} > 2c'r^{-10r}m^{r-1}$ edges as in (\ref{ebound}),
with parts $X$ and $Y$, such that every edge has one vertex in $X$. For each $(r - 1)$-set in $Y$ contained in at most $M - 1$ edges of $G$, remove
all edges of $G$ containing that $(r - 1)$-set. The number of edges that we removed is at most $Mm^{r-1}$, so the remaining  $r$-graph $G' \subset G$ has more than
$$(2c' r^{-10r}-M) m^{r - 1} = c \, m^{r-1}$$
edges. By averaging, there exists
a vertex $x \in X$ whose link $(r - 1)$-graph $G''$ has more than $cm^{r - 2}$ edges. Then $G''$ contains a member $F$ of $\ord(\mathcal{F})$.
Since every edge of $F$ is contained in at least $M$ edges of $G$, we can expand the edges of $F$ to distinct vertices of $X$ to obtain a copy of
$F^+$ in $H$.  \qed

\bigskip

{\bf Proof of Theorem \ref{r-forest}.} We first present an easy proof for $r=2$, and then a significantly more involved general proof.

{\bf Case 1:} $r=2$.
Suppose that $F$ is a forest with $k$ edges. By adding edges, we may assume that $F$ is a tree. We  prove by induction on $k$ that $\ex_{\to}(n, \ord(F)) \le 2k^2n$. Let $H$ be an ordered $n$-vertex graph with more than $2k^2n$ edges and  let $F'$ be
a tree obtained from $F$ by deleting a leaf $y$. Let $x \in V(F')$ be the neighbor of $y$. For each vertex $v$ of $H$, mark the $k$ smallest neighbors of $v$ and the $k$ largest neighbors of $v$. Note that if $v$ has fewer than $k$ smaller neighbors then we mark them all, and similarly for larger neighbors. We marked at most $2kn$ edges so the resulting unmarked graph $H'\subset H$ has more than $2k^2 n - 2kn \ge 2(k-1)^2 n$ edges. By induction, $H'$ contains an interval  $2$-partite subgraph $K'$ isomorphic to $F'$, with parts $A <B$. Suppose that $v$ is the vertex of $K'$ that plays the role $x$ in $F'$, and assume first that $v \in A$. Then there is a vertex $w \in B$ with $\{v,w\} \in K'$, so by construction of $H'$, there is another vertex $w'>w$ such that the marked edge $\{v,w'\} \in H$ and $w' \not\in V(K')$. Adding edge $\{v, w'\}$ to $K'$ gives a copy $K$ of the $2$-interval-partite graph $F$ ($w'$ plays the role of $y$).  The same argument applies if $v \in B$. \qed

\medskip

%{\bf Proof of Theorem \ref{r-forest}.}
{\bf Case 2:} $r\geq 3$. By Theorem \ref{splitting} with $\alpha = r - 1 = k$, it is enough to prove Theorem \ref{r-forest} for interval $(r-1)$-partite $r$-graphs.
Let $H$ be an interval $(r-1)$-partite $r$-graph with $n$ vertices and a partition of $V(H)$ into intervals $X_1 < X_2 < \dots < X_{r-1}$ where for some $i$,
and every $e \in H$, $|e \cap X_i| = 2$ and $|e \cap X_j| = 1$ for $j \neq i$. It is easy to check that every  forest $F$ is contained in a tight tree $T$  with the same set of vertices.
 We show by induction on $t = v(T) \geq r$ that if $e(H) > 2t^2 {n \choose r - 1}$, then $H$ contains a member of $\mbox{ord}(T)$. If $t = r$, then $T$ has one edge and clearly $e(H) = 1$ if $H$ is an $\mbox{ord}(T)$-free. Suppose the statement is true for all tight trees with fewer than $t$ vertices,
and let $T$ be a tight tree with $t$ vertices. Let $H$ be an $n$-vertex interval $(r-1)$-partite $r$-graph  with  more than $2t^2 {n \choose r - 1}$ edges. For each $f \in \partial H$, let $S(f)$ and $L(f)$ denote the set of the $t$ smallest and $t$ largest vertices $x \in V(H)$ such that $f \cup \{x\} \in H$. Then we remove all edges $f \cup \{x\}$
from $H$ such that $x \in S(f) \cup L(f)$. We obtain a new  ordered interval $(r-1)$-partite $r$-graph $H'$ with parts $X_1 < X_2 < \dots < X_{r - 1}$. Let $T' = T - \{y\}$ where
$y$ is a leaf of $T$, and $f \cup \{y\} \in T$. By induction,
$H'$ contains a member of $\mbox{ord}(T')$, since
\[ e(H') >2 t^2 {n \choose r - 1} - 2t {n \choose r - 1} > 2(t - 1)^2 {n \choose r - 1}.\]
Let this member of $\mbox{ord}(T')$ be denoted by $S$, and have parts $A_0 < A_1 < \dots < A_{r-1}$, where $A_{i-1},A_i \subseteq X_i$ and $A_j \subseteq X_j$ for $j \neq i$. Since $f \in \partial S \cap \partial H'$, $f \cap A_j = \emptyset$ for some $j \leq r$. If $j \not \in \{i,i-1\}$, then $S(f) \cup L(f) \subset X_j$, and since $|S(f) \cup L(f)| > t$ and $|V(T) \cap X_j| < t$, there exists $x \in X_j \backslash V(T)$ such that
$f \cup \{x\} \in H$ together with $S$ forms a copy of $T$ in $H$, with interval coloring $A_0' < A_1' < \dots, < A_{r-1}'$ where $A_h' = A_h$ for $h \neq j$ and $A_j' = A_j \cup \{x\}$. If $j = i$, then $f \cup \{z\} \in S$ for some $z \in A_i$. For every $x \in L(f)$, we have $x > z$ and $x \in X_i$. Since $|L(f)| = t$,
there exists $x \in L(f)$ such that $x > z$ and $x \not \in V(S)$. Now $f \cup \{x\} \in H$ together with $S$ is a copy of an element of $\mbox{ord}(T)$ in $H$,
with interval $r$-coloring $A_0' < A_1' < \dots, < A_{r-1}'$ where $A_h' = A_h$ for $h \neq i$ and $A_i' = A_i \cup \{x\}$. Finally,
if $j = i-1$, then $f \cup \{z\} \in S$ for some $z \in A_0$. For every $x \in S(f)$, we have $ x< z$ and $x \in X_{i-1}$. Since $|S(f)| = t$,
there exists $x \in S(f)$ such that $x < z$ and $x \not \in V(S)$. Now $f \cup \{x\} \in H$ together with $S$ is a copy of an element of $\mbox{ord}(T)$ in $H$,
with interval $r$-coloring $A_0' < A_1' < \dots, < A_{r-1}'$ where $A_h' = A_h$ for $h \neq i - 1$ and $A_{i-1}' = A_{i-1} \cup \{x\}$. This completes the proof. \qed

\medskip

{\bf Proof of Theorem~\ref{thmrsz}.} By Theorem~\ref{splitting} with $k = \alpha=2$, it is enough to prove Theorem~\ref{thmrsz} for interval $2$-partite $3$-graphs. Suppose that $\epsilon>0$ and $n_0$ is sufficiently large. Let $H$ be an $n$-vertex ordered interval $2$-partite 3-graph with at least $\epsilon n^2$ edges ($n>n_0$) containing no member of $\ord(I_2)$ and $A<B$ be intervals where every edge of $H$ has exactly one vertex in $A$. Let $G$ be the graph with vertex set $V(H) \cap B$ and edge set $\{yz: \exists x \in A, xyz \in H\}$. Since $H$ contains no member of $\ord(I_2)$, $e(G) =e(H)\ge \epsilon n^2$.   By Theorem B, there is an interval 2-partite  subgraph $G' \subset G$ with at least $\delta n^2$ edges, for some $\delta$ depending only on $\epsilon$.  Consequently,
there is an interval 3-partite subgraph $H' \subset H$  with $\delta n^2$ edges and we apply the Ruzsa-Szemer\'edi Theorem to $H'$ to obtain a copy of some member of $\ord(T_3)$. \qed

\medskip

{\bf Proof of Theorem~\ref{exactell}.}
We use the result of Frankl and F\"{u}redi~\cite{FF} stating that for $0 \leq \ell \leq r - 1$ and some constant $C(r,\ell) > 0$,
\begin{equation}\label{franklfuredi}
\ex(n,I^r(\ell)) < C(r,\ell) \cdot n^{\max\{\ell,r-\ell-1\}}.
\end{equation}
Construction 4 gives a lower bound of order $n^{\alpha}$ for $\ex_{\to}(n,\mbox{ord}(I^r(\ell)))$, 
so it remains to prove the upper bound in Theorem \ref{exactell}. We first prove the upper bound when $\ell$ is odd. 

\medskip

 Recall $\alpha = \max\{\ell, r - (\ell + 1)/2\}$, and let $k = \alpha$, $\ell' = \ell - r + k\ge 0$. Let $H$ be an ordered $n$-vertex $r$-graph with $C(k,\ell')(kn)^{\alpha}/c$ edges, where $c$ is the implicit constant in the second inequality of Theorem \ref{splitting}, namely (\ref{cvalue}). We aim to show $H$ contains a member of $\mbox{ord}(I^r(\ell))$.
By Theorem \ref{splitting} with $k = \alpha$, there is for some $m \in [n]$ an interval $k$-partite subgraph $H'$ of $H$ with $e(H') \geq C(k,\ell')(km)^{\alpha}$ and parts of size at most $m$. For each edge $e \in H'$, 
\[\sum_{j = 1}^k (|e \cap I_j| - 1) = r - k.\]
Let $f(e)$ be the set of the first $|e \cap I_j| - 1$ elements of $e \cap I_j$ for $1 \leq j \leq k$, so that $|f(e)| = r - k$. By the pigeonhole principle, there exists a set $S$ of size $r - k$ such that $f(e) = S$ for at least $|H'|/m^{r-k} \geq C(k,\ell') (km)^{\alpha - r + k}$ edges $e \in H'$.
Let $H'' = \{e \backslash S : S \subset e \in H'\}$, so $H''$ is an ordered $k$-uniform $k$-partite hypergraph with $N = v(H'') \leq km$ and $e(H'') \geq C(k,\ell')N^{\alpha - r + k}$. Since  $2\alpha= \max\{2\ell, 2r - \ell - 1\}\ge 2r-\ell-1$,
\begin{eqnarray*}
	\max\{k-\ell'-1,\ell'\} &=& \max\{r - \ell - 1,\ell - r + k\} \\
	&=& \max\{2r - \ell - 1 - k,\ell\} - r + k \\
	&\le & \max\{2\alpha - k,\ell\} - r + k \; \; = \; \; \max\{\alpha,\ell\}  - r + k \; \; = \; \; \alpha - r + k.
\end{eqnarray*}
It follows from (\ref{franklfuredi}) that $e(H'') \geq C(k,\ell')N^{\alpha - r + k} > \ex(N,I^k(\ell'))$. Therefore there exist $f,g \in H''$ with $|f \cap g| = \ell'$. Since $H''$ is $k$-partite, $\{f,g\} \in \mbox{ord}(I^k(\ell'))$ and now $\{f \cup S,g \cup S\} \in \mbox{ord}(I^r(\ell))$. 
We conclude
\[ \ex_{\to}(n,\mbox{ord}(I^r(\ell))) <  \frac{C(k,\ell')}{c(\alpha,k,r)} (kn)^{\alpha}.\]
This completes the proof of Theorem \ref{exactell} when $\ell$ is odd. 

\medskip

When $\ell\ge 2$ is even, $\alpha = \max\{\ell, r - (\ell + 2)/2\}$. Let $k = \alpha +1$, $\ell' = \ell - r + k \ge 0$, and let $H$ be an ordered $n$-vertex $r$-graph with $C(k,\ell')(kn)^{\alpha}(1+\log_2n)/c$  edges
where $c$ is the implicit constant in the first inequality of Theorem \ref{splitting}. Then for some $m \in [n]$ there is an interval $k$-partite subgraph $H'$ of $H$ with $e(H') \geq C(k,\ell')(km)^{\alpha}$ and parts of size at most $m$. Define the interval $k$-partite $k$-graph $H'' \subseteq H'$ as above. Since $\ell$ is even, $2\alpha= \max\{2\ell, 2r - \ell - 2\}\ge 2r-\ell-2$, and therefore
\begin{eqnarray*}
	\max\{k-\ell'-1,\ell'\} &=& \max\{r - \ell - 1,\ell - r + k\} \\
	&=& \max\{2r - \ell - 1 - k,\ell\} - r + k \\
	&\le& \max\{2\alpha - k + 1,\ell\} - r + k \; \; = \; \; \max\{\alpha,\ell\}  - r + k \; \; = \; \; \alpha - r + k.
\end{eqnarray*}
In the last line we used $k = \alpha + 1$. It follows from (\ref{franklfuredi}) that $H''$ contains a member of $I^k(\ell')$ and then 
$H$ contains a member of $I^r(\ell)$. This completes the proof of Theorem \ref{exactell} when $\ell$ is even. \qed

\paragraph{Acknowledgment.}

This research was partly conducted  during AIM SQuaRes (Structured Quartet Research Ensembles) workshops, and we gratefully acknowledge the support of AIM. We are very grateful to the referees for their careful reading of the paper that helped tremendously in improving the paper. We are especially thankful for spotting an important error in  our main theorem in the first version of this paper. In correcting the error we managed to  generalize   our main result and shorten its proof.

\end{document}